\documentclass[12pt, a4paper]{article}
\usepackage{ngerman}
\usepackage{amsmath}
\usepackage{amsthm}
\usepackage{amsfonts}
\usepackage{amssymb}
\usepackage{graphicx}
\usepackage{geometry}
\usepackage{a4wide}
\usepackage{longtable}
\usepackage[latin1]{inputenc}

\DeclareMathOperator{\modulo}{\ mod }
\DeclareMathOperator{\SL}{SL}

\DeclareMathOperator{\II}{\textit{II}}

\DeclareMathOperator{\Z}{\mathbb{Z}}
\DeclareMathOperator{\R}{\mathbb{R}}
\DeclareMathOperator{\C}{\mathbb{C}}
\DeclareMathOperator{\Q}{\mathbb{Q}}

\DeclareMathOperator{\real}{Re}

\DeclareMathOperator{\Number}{N}
\DeclareMathOperator{\Gr}{Gr}
\DeclareMathOperator{\e}{\mathfrak{e}}

	\newtheorem{Satz}{Satz}[section]
	\newtheorem{theorem}[Satz]{Theorem}
	\newtheorem{lemma}[Satz]{Lemma}
	\newtheorem{proposition}[Satz]{Proposition}

	\theoremstyle{definition}

\date{}
\author{Moritz Dittmann, Heike Hagemeier and Markus Schwagenscheidt}
\title{Automorphic products of singular weight for simple lattices}

\begin{document}
\maketitle

\begin{abstract}
We classify the simple even lattices of square free level and signature $(2,n), n \geq 4$. A lattice is called simple if the space of cusp forms of weight $1+n/2$ for the dual Weil representation of the lattice is trivial. For a simple lattice every formal principal part obeying obvious conditions is the principal part of a vector valued modular form. Using this, we determine all holomorphic Borcherds products of singular weight (arising from vector valued modular forms with non-negative principal part) for the simple lattices. We construct the corresponding vector valued modular forms by eta products and compute expansions of the automorphic products at different cusps.
\end{abstract}

\section{Introduction}

	The singular theta correspondence allows the construction of automorphic forms for orthogonal groups from vector valued modular forms for $\SL_{2}(\mathbb{Z})$. Automorphic forms arising in this way have nice infinite product expansions similar to the Dedekind eta function. They are therefore called automorphic products or Borcherds products.
	
	In general, non-trivial holomorphic automorphic products do not exist for arbitrary weights. The smallest weight which can occur is called singular weight. One nice feature of holomorphic automorphic products of singular weight is that many of their Fourier coefficients vanish, which makes it possible to determine their Fourier expansions in some cases in a combinatorial way.
	
	In \cite{ScheithauerClassification}, Scheithauer classified reflective automorphic products of singular weight for lattices of square free level. The aim of the present work is the classification of holomorphic automorphic products of singular weight for so-called simple lattices of square free level (arising from vector valued modular forms with non-negative principle part).
	 
	 We describe our results in more detail: 
	
Let $L$ be an even lattice of signature $(2,n), n \geq 4$ even, with dual lattice $L'$. Let $\rho_L$ be the corresponding Weil representation and let $F$ be a weakly holomorphic vector valued modular form for $\rho_L$ of (negative) weight $k=1-n/2$ with component functions  	$f_{\gamma}(\tau) = \sum_{\substack{m  \in \Q}}c(\gamma,m)e(m\tau)$. Suppose $c(\gamma,m)\in\Z$ for $m<0$. The singular theta correspondence (see {\cite{Borcherds}},{\cite{BruinierBorcherdsProducts}}) sends $F$ to a meromorphic function $\Psi$ on the Grassmannian $\Gr(L)$ transforming as an automorphic form for the discriminant kernel $\Gamma(L)$ of weight $c(0,0)/2$. 

Pairing the form $F$ with a vector valued cusp form of (positive) weight $2-k$ for the dual Weil representation gives a meromorphic elliptic modular form of weight $2$ for $\SL_{2}(\Z)$, hence its constant term must vanish by the residue theorem. This gives conditions on the principal part of $F$. In fact there are no further conditions, so if the space of cusp forms above is trivial then any principal part occurs. 

This observation can be used to construct automorphic products. Namely, if the space of cusp forms of weight $1+n/2$ for the dual Weil representation of $L$ is trivial (such lattices will be called simple), then there exist modular forms of weight $1-n/2$ for the Weil representation for any given principal part, which can then be lifted to automorphic products. It turns out that for square free level there are only $15$ simple lattices up to isomorphism (cf. Theorem \ref{EinfacheGitter}):
\begin{longtable}{|c|c|c|}
		\hline
		$\text{Level}$ & $\text{Genus}$ & $\text{Lattice}$ \\
		\hline\hline
		$1$ & $\II_{2,10}$ &  \\
		 & $\II_{2,18}$ & \\
		 & $\II_{2,26}$ &  \\
		\hline
		$2$ & $\II_{2,6}(2_{\II}^{-2})$ & $D_{4}(-1) \oplus \II_{1,1} \oplus \II_{1,1}$ \\
		& $\II_{2,6}(2_{\II}^{-4})$ & $D_{4}(-1) \oplus \II_{1,1}(2) \oplus \II_{1,1}$\\
		& $\II_{2,6}(2_{\II}^{-6})$ & $D_{4}(-1) \oplus \II_{1,1}(2) \oplus \II_{1,1}(2)$ \\
		& $\II_{2,10}(2_{\II}^{+2})$ & $E_{8}(-1) \oplus \II_{1,1}(2) \oplus \II_{1,1}$ \\
		\hline
		$3$ & $\II_{2,4}(3^{+1})$ & $A_{2}(-1) \oplus \II_{1,1} \oplus \II_{1,1}$ \\
		& $\II_{2,4}(3^{-3})$ & $A_{2}(-1) \oplus \II_{1,1}(3) \oplus \II_{1,1}$ \\
		& $\II_{2,4}(3^{+5})$ & $A_{2}(-1) \oplus \II_{1,1}(3) \oplus \II_{1,1}(3)$ \\
		& $\II_{2,8}(3^{-1})$ & $E_{6}(-1) \oplus \II_{1,1} \oplus \II_{1,1}$\\
		\hline
		$5$ & $\II_{2,6}(5^{+1})$ & $A_{4}(-1) \oplus \II_{1,1} \oplus \II_{1,1}$ \\
		\hline
		$6$ & $\II_{2,4}(2_{\II}^{+2}3^{+1})$ & $A_2(-1) \oplus \II_{1,1}(2) \oplus \II_{1,1}$\\
		 & $\II_{2,4}(2_{\II}^{+4}3^{+1})$ & $A_2(-1) \oplus \II_{1,1}(2) \oplus \II_{1,1}(2)$\\
		\hline
		$7$ & $\II_{2,8}(7^{+1})$ & $A_{6}(-1) \oplus \II_{1,1} \oplus \II_{1,1}$\\
		\hline
\end{longtable}
The weight of an automorphic product is completely determined by the corresponding vector valued modular form. We are interested in automorphic products of so-called singular weight, which is given by $k=n/2-1$. We also want them to be holomorphic, so we assume that all coefficients of the principal part of the corresponding vector valued modular form are non-negative. This is not a necessary condition, but for simplicity we are only considering vector valued modular forms with this property. In section \ref{sec:Existence} we show that the only holomorphic automorphic products of singular weight on simple lattices coming from vector valued modular forms with non-negative principal part are the following (cf. Theorem \ref{thm:ProdukteAufEinfachenGittern}). 
	\textit{
		\begin{itemize}
			\item $\II_{2,26}$: The singular weight is $12$ and the corresponding modular form $F$ has principal part $e(-\tau)$.
			\item $\II_{2,6}(2_{\II}^{-6})$: The singular weight is $2$ and the corresponding modular form $F$ has principal part $e(-\tau/2)\e_{\gamma}$ for some $\gamma \in L'/L$ with $Q(\gamma) = 1/2 \mod 1$.
			\item $\II_{2,10}(2_{\II}^{+2})$: The singular weight is $4$ and the corresponding modular form $F$ has principal part $e(-\tau/2)\e_{\gamma}$
			for some $\gamma \in L'/L$ with $Q(\gamma) = 1/2 \mod 1$.
			\item $\II_{2,4}(3^{+5})$: The singular weight is $1$ and the corresponding modular form $F$ has principal part $e(-\tau/3)\e_{\gamma}+ e(-\tau/3)\e_{-\gamma}$ for some $\gamma \in L'/L$ with $Q(\gamma) = 2/3 \mod 1$.
		\end{itemize}}

	The automorphic product of weight $12$ for $\II_{2,26}$ is well known and can be constructed as the Borcherds product of $1/\Delta$ (see  \cite{BorcherdsInfiniteProducts}). The product for the lattice of level $2$ and signature $(2,10)$ has already been described in \cite{Borcherds}, example 13.7, and \cite{ScheithauerWeil}, section 8. The other two products seem to be new. In the last section we explicitly construct them as lifts of eta products (cf. Theorem \ref{thm:Konstruktion}):
	\textit{\begin{enumerate}
\item
The vector valued modular form corresponding to the case $L'/L=\II_{2,4}(3^{+5})$ can be obtained as $F=F_{f,\Gamma_1(3),\gamma}$  
where 
\[
f=\frac{\eta(\tau)}{\eta(3\tau)^3}
\] 
and $\gamma\in L'/L$ is an element of norm $Q(\gamma)\equiv -1/3\mod 1$. 
\item The vector valued modular form corresponding to the case $L'/L=\II_{2,6}(2_{\II}^{-6})$ can be obtained as $F=F_{f,\Gamma_1(2),\gamma}$ where
\[
f=\frac{\eta(\tau)^4}{\eta(2\tau)^8}
\] 
and $\gamma\in L'/L$ is an element of norm $Q(\gamma)\equiv 1/2\mod 1$. 
\end{enumerate}}

 Note that in contrast to most of the known examples, the corresponding modular forms $F$ for the new products are not symmetric, i.e. not invariant under the action $F^{\sigma} = \sum_{\gamma}f_{\gamma}\e_{\sigma(\gamma)}$ of the orthogonal group.
	
The rest of the last section is dedicated to computing the Fourier coefficients of these two automorphic products at various cusps.	

	\subsection*{Acknowledgements}

	The work of this paper extends the doctoral thesis of the second author \cite{Hagemeier}. The second author likes to thank the supervisor of this thesis, Jan Bruinier, for his support and encouragement.
	
		Moreover, we thank Nils Scheithauer who supervised the master's theses of the first and the third author and suggested the topic of this paper. Without his constant support and effort this work would not have been possible.

	We also like to thank Sebastian Opitz who performed the necessary computer calculations in the determination of the simple lattices in section $8$.

	\section{Dimension formulas for spaces of vector valued modular forms}

		We recall the definition of vector valued modular forms for the Weil representation from \cite{BruinierBorcherdsProducts}:
	
	Let $L$ be an even lattice of signature $(b^{+},b^{-})$ with quadratic form $q$ and associated bilinear form $\langle \cdot , \cdot\rangle$. We assume that the difference $r = b^{+} - b^{-}$ is even. This condition is satisfied by all lattices of square free level. Later we will only consider lattices of type $(b^{+},b^{-}) = (2,n)$ with $n \geq 4$ even, so the condition on $r$ is no restriction. We use the notation $e(z) := e^{2\pi i z}$.
	
	Let $L'$ be the dual lattice of $L$ and let $N$ be the level of $L$, i.e. the smallest positive integer such that $Nq(\gamma) \in \Z$ for all $\gamma \in L'$. The discriminant form $D = L'/L$ is a finite abelian group of order $|\det(L)|$ and the modulo $1$ reduction of the quadratic form $q$ induces a $\Q/\Z$-valued quadratic form $Q$ and a corresponding bilinear form $(\cdot,\cdot)$ on $D$. Note that since $NL' \subseteq L$ every element of $D$ has order dividing $N$. 
	
	The group ring $\mathbb{C}[D]$ is the set of all formal linear combinations $\sum_{\gamma \in D}a_{\gamma}\e_{\gamma}$ with $a_{\gamma} \in \mathbb{C}$ and formal basis vectors $\e_{\gamma}$. The Weil representation $\rho_{L}$ of $\SL_{2}(\Z)$ on $\C[D]$ is defined on the generators $T = \left( \begin{smallmatrix} 1 & 1 \\ 0 & 1 \end{smallmatrix}\right)$ and $S = \left(\begin{smallmatrix}0 & -1 \\1 & 0 \end{smallmatrix} \right)$ by
	\begin{align*}
	\rho_{L}(T)\e_{\gamma} &= e(Q(\gamma))\e_{\gamma}, \\
	\rho_{L}(S)\e_{\gamma} &= \frac{e(-r/8)}{\sqrt{|D|}}\sum_{\beta \in D}e(-(\beta,\gamma))\e_{\beta}.
	\end{align*}
	For example, a matrix $M \in \Gamma_{1}(N)$ acts as $\rho_{L}(M)\e_{\gamma} = e(bQ(\gamma))\e_{\gamma}$ (see \cite{ScheithauerWeil}). The dual representation will be denoted by $\rho_{L}^{*}$.
	
	 A weakly holomorphic vector valued modular form for the Weil representation $\rho_{L}$ of weight $k \in \mathbb{Z}$ is a holomorphic function $F = \sum_{\gamma \in D}f_{\gamma}\e_{\gamma}: \mathbb{H} \to \mathbb{C}[D]$ which transforms as $F(M\tau) =(c\tau + d)^{k}\rho_{L}(M)F(\tau)$ for $M = \left( \begin{smallmatrix} a & b \\ c & d \end{smallmatrix}\right) \in \SL_{2}(\mathbb{Z})$ and whose component functions $f_{\gamma}$ have Fourier expansions of the form
	\[
	f_{\gamma}(\tau) = \sum_{\substack{m \in \mathbb{Z} + Q(\gamma) \\ m \gg -\infty}}c(\gamma,m)e(m\tau ), \qquad c(\gamma,m) \in \mathbb{C}.
	\]
	Modular forms for the dual representation $\rho_{L}^{*}$ are defined analogously.
	
	A modular form $F$ for $\rho_{L}$ is called a holomorphic modular form (resp. cusp form) if the coefficients $c(\gamma,m)$ in the Fourier expansions of the $f_{\gamma}$ vanish for all $m < 0$ (resp. for all $m \leq 0$). The corresponding spaces are denoted by $M_{k,\rho_{L}}$ and $S_{k,\rho_{L}}$. Since there are no holomorphic modular forms of negative weight, a weakly holomorphic modular form of negative weight is completely determined by its principal part, i.e. the finite sum
	\begin{equation}\label{principalpart}
	\sum_{\gamma\in D}\sum_{m<0}c(\gamma,m)e(m\tau)\e_\gamma.
	\end{equation}

	Using the Selberg trace formula or the Riemann-Roch theorem one can show that the dimension of the space of holomorphic modular forms $M_{k,\rho_{L}^{*}}$ for the dual Weil representation $\rho_{L}^{*}$ of weight $k$ with $k\geq 2$ equals
		\begin{equation*}
		\dim(M_{k,\rho_{L}^{*}}) = d+\frac{dk}{12}-\alpha(e(k/4)S)-\alpha((e(k/6)ST)^{-1})-\alpha(T),
		\end{equation*}  
		where $d$ is the dimension of the biggest subspace $V_0$ of $\C[D]$ on which $-I=S^2$ acts by multiplication with $e(-k/2)$ and $\alpha(X)$ is the sum of the numbers $\beta_j, 1\leq j\leq d$, where the eigenvalues of $X$ are $e(\beta_j)$ and $0\leq\beta_j<1$ (cf. \cite{BorcherdsGKZ}, \cite{Freitag}).
		
		 The following theorem gives an explicit dimension formula in terms of the Gauss sums $G(n,L)=\sum_{\gamma\in D}e(nQ(\gamma))$ and the numbers of elements of a given norm in the discriminant form of $L$.
		
	\begin{theorem}
		Let $L$ be an even lattice of signature $(b^+,b^-)$, level $N$ and even $r = b^{+}-b^{-}$, and let $k\geq 2$ be an integer. Then the dimension of $M_{k,\rho_{L}^*}$ is given by
		\begin{align*}
		\dim(M_{k,\rho_{L}^*})= d + \frac{dk}{12}- \alpha_{1} - \alpha_{2} - \alpha_{3}		 
		\end{align*}
		with
		\begin{align*}
		d &= |D|/2+c|D^2|/2, \\
		\alpha_1&=\frac{d}{4}-\frac{1}{4\sqrt{|D|}}e((2k+b^+-b^-)/8)(G(2,L)+cG(-2,L)),\\
		\alpha_2&=\frac{d}{3}+\frac{1}{3\sqrt{3|D|}}\real(e((4k+3b^+-3b^--10)/24)(G(1,L)+cG(-3,L))),\\
		\alpha_3&= \sum_{j=1}^N \frac{j}{2N}\Number(D,j)-c\sum_{j=1}^4 \frac{j}{8}\Number(D^2,j).
		\end{align*}
		where $c=(-1)^{(2k+b^+-b^-)/2}$ and $N(D,j)$ is the number of elements $\gamma \in D$ of norm $Q(\gamma) = j/N \mod 1$ and $D^n=\{\gamma\in D:n\gamma=0\}$.
		\end{theorem}

	\begin{proof}
		The space $V_0$ is spanned by $\{\e_\gamma+c\e_{-\gamma}:\gamma\in D\}$, so that $d=|D|/2+c|D^2|/2$.
		We denote the numbers $\alpha(e(k/4)S), \alpha((e(k/6)ST)^{-1})$ and $\alpha(T)$ by $\alpha_1,\alpha_2$ and $\alpha_3$ respectively. The quantities $\alpha_{1},\alpha_{2}$ can be computed as in \cite{BruinierPicardGroups}, Lemma 2. As in the proof of Lemma 4 in \cite{BruinierPicardGroups} it can be seen that
		\begin{align*}
		\alpha_3=\frac{1}{2}\sum_{\gamma\in D}(-Q(\gamma)-\lfloor-Q(\gamma)\rfloor)+\frac{c}{2}\sum_{\gamma\in D^2}(-Q(\gamma)-\lfloor-Q(\gamma)\rfloor).
		\end{align*}
		\end{proof}

		For lattices of square free level, both the Gauss sums and the number of elements of a given norm can easily be computed using the formulas from Theorem 3.9 in \cite{ScheithauerWeil} and section 3 in \cite{ScheithauerClassification}.
		
		We also need formulas for the dimensions of the spaces of cusp forms $S_{k,\rho_{L}^*}$. Using Eisenstein series (cf. \cite{BruinierBorcherdsProducts} chapter 1.2.3) one can show that if $k>2$ then the codimension $\alpha_4$ of $S_{k,\rho^{*}_{L}}$ in $M_{k,\rho_{L}^{*}}$ is equal to 
		\begin{equation*}
		\alpha_{4} = \frac{1}{2}|\{\gamma\in D:Q(\gamma)=0 \modulo 1\}|+\frac{c}{2}|\{\gamma\in D^2:Q(\gamma)=0 \modulo 1\}|.
		\end{equation*}

	\section{Simple lattices of square free level}

		In this section we show that there are only finitely many simple lattices (i.e. lattices of signature $(2,n)$ such that the space of cusp forms of weight $k=1+n/2$ for the dual Weil representation is trivial) of square free level and then determine all of these lattices. In order to do so, we need bounds on the $\alpha_i$. 
		
		Using $|G(n,L)| \leq \sqrt{|D|}\sqrt{|D^{n}|}$ (Lemma 1 in \cite{BruinierPicardGroups}) it is easily seen that
		\begin{align*}
		\left|\alpha_1 - d/4\right|&\leq \frac{\sqrt{|D^2|}}{4}, \\
		\left|\alpha_2 - d/3\right|&\leq \frac{1+\sqrt{|D^3|}}{3\sqrt{3}}.
		\end{align*}
		Estimating $\alpha_3$ and $\alpha_4$ is more difficult and is done in the proofs of the following propositions, which give bounds on the size of the discriminant form of a simple lattice.
\begin{proposition}
Let $L$ be a lattice of square free level and signature $(2,n)$ with $n\geq 8$. Suppose $S_{1+n/2,\rho_L^*}=\{0\}$. Then $|L'/L|<45$.
\end{proposition}	
\begin{proof}
Note that $\alpha_3$ and $\alpha_4$ do not depend on the weight $k$, but only on the parity of $k$. Hence we can apply the dimension formula with $k=n/2-1$ to obtain bounds on $\alpha_3$ and $\alpha_4$. 
\begin{align*}
		\alpha_3+\alpha_4&\leq d+\frac{(n/2-1)d}{12}-\frac{d}{4}+\frac{\sqrt{|D^2|}}{4}-\frac{d}{3} \\
		& \qquad\qquad +\frac{1+\sqrt{|D^3|}}{3\sqrt{3}}-\dim(S_{n/2-1,\rho_{L}^*})\\
										 &\leq \frac{(n/2+4)d}{12}+\frac{\sqrt{|D^2|}}{4}+\frac{1+\sqrt{|D^3|}}{3\sqrt{3}}.
		\end{align*}
The dimension formula for the weight $k=1+n/2$ now gives
\begin{align*}
		\dim(S_{1+n/2,\rho_{L}^*})&\geq \frac{d}{6}-\frac{\sqrt{|D^2|}}{2}-\frac{2+2\sqrt{|D^3|}}{3\sqrt{3}}\\
															&=\frac{|D|}{12}+\frac{\sqrt{|D^2|}}{12}\left(\sqrt{|D^2|}-6\right)-\frac{2+2\sqrt{|D^3|}}{3\sqrt{3}}\\
															&\geq\frac{|D|}{12}-\frac{3}{4}-\frac{2}{3\sqrt{3}}-\frac{2\sqrt{|D|}}{3\sqrt{3}}\\
															&\geq\frac{\sqrt{|D|}}{12}\left(\sqrt{|D|}-\frac{8}{\sqrt{3}}\right)-\frac{3}{4}-\frac{2}{3\sqrt{3}}.
		\end{align*} 
The right hand side is positive for $|D|\geq 45$.
\end{proof}	
		
\begin{proposition}
Let $L$ be a lattice of square free level $N$ and signature $(2,6)$. Suppose $S_{4,\rho_L^*}=\{0\}$. Then $N<33$ or $|L'/L|<137$.
\end{proposition}
\begin{proof}
As mentioned in the proof of the last proposition, $\alpha_3$ and $\alpha_4$ do not depend on the weight $k$ but only on the parity of $k$. We write $\alpha_3^e, \alpha_4^e$ (resp. $\alpha_3^o,\alpha_4^o$) for the values of $\alpha_3$ and $\alpha_4$ if $k$ is even (resp. odd). Note that 
\begin{align*}
		\alpha_3^e+\alpha_4^e-\alpha_3^o-\alpha_4^o\leq |D^2|.
		\end{align*}
The dimension formula for $k=3$ yields	
\begin{align*}
		\alpha_3^o+\alpha_4^o&\leq \frac{1}{3}\left(|D|-|D^2|\right)+\frac{\sqrt{|D^2|}}{4}+\frac{1+\sqrt{|D^3|}}{3\sqrt{3}}-\dim(S_{3,\rho_{L}^*})\\
										 &\leq \frac{1}{3}\left(|D|-|D^2|\right)+\frac{\sqrt{|D^2|}}{4}+\frac{1+\sqrt{|D^3|}}{3\sqrt{3}}.
		\end{align*}
The dimension formula for $k=4$ then gives
		\begin{align*}
		\dim(S_{4,\rho_{L}^*}) &\geq \frac{9}{24}\left(|D|+|D^2|\right)-\frac{\sqrt{|D^2|}}{4}-\frac{1+\sqrt{|D^3|}}{3\sqrt{3}}-\alpha_3^e-\alpha_4^e\\
		&\geq \frac{|D|}{24}-\frac{7}{24}|D^2|-\frac{\sqrt{|D^2|}}{2}-\frac{2+2\sqrt{|D^3|}}{3\sqrt{3}}.
		\end{align*}
		If $N\geq 33$, then $|D^2|\leq|D|/16$ and $|D^3|\leq|D|/11$ and therefore
		\begin{align*}
		\dim(S_{4,\rho_{L}^*})\geq\frac{3}{128}|D|-\frac{\sqrt{|D|}}{8}-\frac{2}{3\sqrt{3}}-\frac{2\sqrt{|D|}}{3\sqrt{33}}.
		\end{align*}
		The right hand side is positive for $|D|\geq 137$.
		\end{proof}	
\begin{proposition}
Let $L$ be a lattice of square free level $N$ and signature $(2,4)$. Suppose $S_{3,\rho_L^*}=\{0\}$. Then $N<101$ or $|L'/L|<3277$.
\end{proposition}
\begin{proof}
Applying the dimension formula for $k=2$ yields
		\begin{align*}
		\alpha_3^e&\leq \frac{7}{24}(|D|-|D^2|)+\frac{\sqrt{|D^2|}}{4}+\frac{1+\sqrt{|D^3|}}{3\sqrt{3}}-\dim(M_{2,\rho_{L}^*}) \\
							&\leq \frac{7}{24}(|D|-|D^2|)+\frac{\sqrt{|D^2|}}{4}+\frac{1+\sqrt{|D^3|}}{3\sqrt{3}}
		\end{align*}
		and thus
		\begin{align}\label{eqn:S_3}
		\dim(S_{3,\rho_{L}^*})\geq \frac{1}{24}|D|-\frac{3}{8}|D^2|-\frac{\sqrt{|D^2|}}{2}-\frac{2+2\sqrt{|D^3|}}{3\sqrt{3}}-\alpha_4.
		\end{align}
Now suppose $N>100$. The following lemma then shows that 
\begin{equation}\label{a4bound}
		\alpha_4\leq \frac{|D^2|}{2}+\frac{|D|}{50}.
		\end{equation}
Note that $|D^2|\leq 64 = 2^{6}$ and $|D^3|\leq 729 = 3^{6}$ as the $p$-ranks of the Jordan components of $D$ are bounded by the rank $2+n=6$ of $L$. Hence,  we can use \eqref{eqn:S_3} to obtain
		\begin{equation*}
		\dim(S_{3,\rho_{L}^*})> \frac{13}{600}|D|-71.
		\end{equation*} 
		This is greater than $0$ for $|D|\geq 3277$. 
\end{proof}
It remains to show equation \eqref{a4bound}.
	\begin{lemma}
		Suppose $L$ is a lattice of signature $(2,4)$ and square free level $N$ with $N>100$. Then 
		\begin{equation*}
		\alpha_4\leq \frac{|D^2|}{2}+\frac{|D|}{50}.
		\end{equation*}
		\end{lemma}
		
		\begin{proof}
		Obviously,
		\begin{equation*}
		\alpha_4\leq \frac{|D^2|}{2}+\frac{1}{2}|\{\gamma\in D: Q(\gamma)=0\mod 1\}|.
		\end{equation*}
		Using propositions 3.1 and 3.2 in \cite{ScheithauerClassification} one can see that the number of elements of norm $0 \modulo 1$ in a Jordan component $p^{\epsilon n_p}$ with $n_p\geq 1$ is less or equal to $2p^{n_p-1}$ if $n_p$ is even and $p^{n_p-1}$ if $n_p$ is odd. If the $p$-ranks of all Jordan components of $D$ were even, the signature of $D$ (which is the signature of $L \mod 8$) would be divisible by $4$ by the oddity formula. Hence there is at least one Jordan component of odd rank. This, together with Proposition 3.3 in \cite{ScheithauerClassification}, implies that 
		\begin{equation*}
		|\{\gamma\in D: Q(\gamma)=0\mod 1\}|\leq 2^{(\omega(N)-1)}\frac{|D|}{N}
		\end{equation*}
		where $\omega(N)$ denotes the number of prime factors of $N$. Thus, if the number of prime factors of $N$ is less than $4$, then 
		\begin{equation*}
		|\{\gamma\in D: Q(\gamma)=0\mod 1\}|\leq \frac{|D|}{25}
		\end{equation*}
		and otherwise
		\begin{equation*}
		|\{\gamma\in D: Q(\gamma)=0\mod 1 \}|\leq \frac{8}{2\cdot 3\cdot 5 \cdot 7}|D|=\frac{4}{105}|D|<\frac{|D|}{25},
		\end{equation*}
		which completes the proof.
		\end{proof}

In all cases either the level or the size of the discriminant form is bounded. But there are only finitely many discriminant forms of a given size. The same is true for a fixed level if one also fixes the signature. Hence the number of simple lattices is finite. 
For his master's thesis \cite{Opitz} Sebastian Opitz wrote a Magma programme that computes the dimensions of the spaces of modular forms and cusp forms of a given weight $k$ for a given discriminant form. Using this programme he checked the remaining cases, giving the following result.
		
 \begin{theorem}\label{EinfacheGitter}
			Every simple even lattice of signature $(2,n), n \geq 4$, of square free level is isomorphic to one of the following lattices:
		\begin{align*}
		\begin{array}{|c|c|c|}
		\hline
		\text{Level} & \text{Genus} & \text{Lattice} \\
		\hline\hline
		1 & \II_{2,10} &  \\
		 & \II_{2,18} & \\
		 & \II_{2,26} &  \\
		\hline
		2 & \II_{2,6}(2_{\II}^{-2}) & D_{4}(-1) \oplus \II_{1,1} \oplus \II_{1,1} \\
		& \II_{2,6}(2_{\II}^{-4}) & D_{4}(-1) \oplus \II_{1,1}(2) \oplus \II_{1,1}\\
		& \II_{2,6}(2_{\II}^{-6}) & D_{4}(-1) \oplus \II_{1,1}(2) \oplus \II_{1,1}(2) \\
		& \II_{2,10}(2_{\II}^{+2}) & E_{8}(-1) \oplus \II_{1,1}(2) \oplus \II_{1,1} \\
		\hline
		3 & \II_{2,4}(3^{+1}) & A_{2}(-1) \oplus \II_{1,1} \oplus \II_{1,1} \\
		& \II_{2,4}(3^{-3}) & A_{2}(-1) \oplus \II_{1,1}(3) \oplus \II_{1,1} \\
		& \II_{2,4}(3^{+5}) & A_{2}(-1) \oplus \II_{1,1}(3) \oplus \II_{1,1}(3) \\
		& \II_{2,8}(3^{-1}) & E_{6}(-1) \oplus \II_{1,1} \oplus \II_{1,1}\\
		\hline
		5 & \II_{2,6}(5^{+1}) & A_{4}(-1) \oplus \II_{1,1} \oplus \II_{1,1} \\
		\hline
		6 & \II_{2,4}(2_{\II}^{+2}3^{+1}) & A_2(-1) \oplus \II_{1,1}(2) \oplus \II_{1,1}\\
		 & \II_{2,4}(2_{\II}^{+4}3^{+1}) & A_2(-1) \oplus \II_{1,1}(2) \oplus \II_{1,1}(2)\\
		\hline
		7 & \II_{2,8}(7^{+1}) & A_{6}(-1) \oplus \II_{1,1} \oplus \II_{1,1} \\
		\hline
		\end{array}
		\end{align*}
		\end{theorem}
	
		Note that each genus in the above list contains exactly one isomorphism class (see \cite{ConwaySloane}, chapter 15, Corollary 22).
		
	\section{Existence of automorphic products of singular weight}\label{sec:Existence}
Let $L$ be an even lattice of signature $(2,n)$,$n \geq 4$ even, and let $F$ be a weakly holomorphic modular form of weight $k=1-n/2$ for the Weil representation of $L$ with $c(\gamma,m)\in\Z$ for $m<0$. Borcherds' singular theta correspondence (see \cite{Borcherds},\cite{BruinierBorcherdsProducts}) sends $F$ to a meromorphic function $\Psi$ on the Grassmannian $\Gr(L)$ (consisting of the 2-dimensional positive definite subspaces of $L\otimes\R$) transforming as an automorphic form for the discriminant kernel $\Gamma(L)$ of weight $c(0,0)/2$. Moreover, if the coefficients $c(\gamma,m)$ of $F$ are non-negative for $m<0$, then $\Psi$ will be holomorphic. 

In order to construct such functions $\Psi$, we have to find suitable vector valued modular forms $F$, which in turn are given by their principal parts. A finite sum as in $\eqref{principalpart}$ is the principal part of a vector valued modular form $F$ of weight $k=1-n/2$ for $\rho_L$ if and only if $c(\gamma,m)=c(-\gamma,m)$ for all $\gamma$ and $m$ and
\begin{align}\label{obstruction}
\sum_{\gamma\in D}\sum_{m<0}c(\gamma,m)a(\gamma,-m)=0
\end{align}
for every cusp form $\sum_{\gamma,m}a(\gamma,m)e(m\tau)\e_\gamma$ of weight $1 + n/2$ for the dual Weil representation (cf. Theorem 1.17 in \cite{BruinierBorcherdsProducts}).

We want $\Psi$ to have singular weight $n/2-1$, i.e. we want the coefficient $c(0,0)$ of $F$ to be $n-2$. The weight of $\Psi$ is given by
 \begin{align}\label{weightautomorphicproduct}
	 -\frac{1}{4}\sum_{\gamma \in D}\sum_{m < 0}c(\gamma,m)q(\gamma,-m)
	 \end{align}
	 where the $q(\gamma,-m)$ are the coefficients of a vector valued Eisenstein series of weight $1 + n/2$ for the dual Weil representation $\rho_{L}^{*}$ (cf. \cite{BruinierKuss}, Theorem 5). In the case of square free level Scheithauer has calculated the Eisenstein coefficients (see \cite{ScheithauerClassification}, Theorem 7.1).
	One can now try to adjust the coefficients $c(\gamma,m)$ for $m < 0$ in order to obtain automorphic products of singular weight. If $L$ is simple, the only conditions on the principal part are $c(\gamma,m) = c(-\gamma,m)$ and $c(\gamma,m) \in \Z$ for $m < 0$. We find the following automorphic products of singular weight:
	
	
\begin{theorem}\label{thm:ProdukteAufEinfachenGittern}Automorphic products (coming from vector valued modular forms with non-negative principal part) of singular weight for simple lattices of square free level only exist in the following cases:
		\begin{itemize}
			\item $\II_{2,26}$: The singular weight is $12$ and the corresponding modular form $F$ has principal part $e(-\tau)$.
			\item $\II_{2,6}(2_{\II}^{-6})$: The singular weight is $2$ and the corresponding modular form $F$ has principal part $e(-\tau/2)\e_{\gamma}$ for some $\gamma \in D$ with $Q(\gamma) = 1/2 \mod 1$.
			\item $\II_{2,10}(2_{\II}^{+2})$: The singular weight is $4$ and the corresponding modular form $F$ has principal part $e(-\tau/2)\e_{\gamma}$
			for some $\gamma \in D$ with $Q(\gamma) = 1/2 \mod 1$.
			\item $\II_{2,4}(3^{+5})$: The singular weight is $1$ and the corresponding modular form $F$ has principal part $e(-\tau/3)\e_{\gamma}+ e(-\tau/3)\e_{-\gamma}$ for some $\gamma \in D$ with $Q(\gamma) = 2/3 \mod 1$.
		\end{itemize}
	\end{theorem}
\begin{proof}
We need the formula for the coefficients of the Eisenstein series. This is quite lengthy, so we will not reproduce it here but refer the reader to \cite{ScheithauerClassification}, Theorem 7.1. These formulas can be considerably simplified for each of the simple lattices separately. For lattices of odd prime level the simplified formula is written down in \cite{Hagemeier}, Proposition 5.31.
We will repeatedly make use of the estimate
\begin{align}\label{DivisorSumEstimate}
	m^{k}(2-\zeta(k)) \leq \sum_{d|m}a_{d}d^{k} \leq m^{k}\zeta(k)
	\end{align}
	which holds for $a_{d} \in \{-1,0,1\}, a_{m} = 1, k \geq 2$. Here $\zeta$ is the Riemann zeta function.
	The proof is similar for each of the simple lattices so we will only show some of the necessary calculations.

	For example, let $L$ be the simple lattice of level $3$, signature $(2,4)$ and discriminant group $D \cong (\mathbb{Z}/3\mathbb{Z})^{5}$. The Eisenstein coefficient $q(\gamma,m)$ for $m > 0$, $m\equiv - Q(\gamma) \mod 1$ is given by
		\begin{align*}
		q(\gamma,m) = -2\sum_{d|3m}\chi(3m/d)d^{2}
		\end{align*}
		for $\gamma \neq 0$ and
		\begin{align*}
		q(0,m) = -18\sum_{d|m}\chi(m/d)d^{2} - 18\sum_{d|m}\chi(d)d^{2}
		\end{align*}
		for $\gamma = 0$, where $\chi(d) = \left(\frac{d}{3} \right)$ is the Legendre symbol. We list the first few coefficients:
		$$
		\begin{array}{c|c|c|c|c|c|c|c|c|c|c|c|c}		
		m & 1/3 & 2/3 & 1 & 4/3 & 5/3 & 2 & 7/3 & 8/3 & 3 & 10/3 & 11/3 & 4 \\
		\hline
		q(\gamma,m) & -2 & -6 & -18 & -26 & -48 & - 54 & -100 & -102 & -162 & -144 & -240 & -234 \\
		q(0,m) &  &  & -36 &  &  & 0  & & & -180 & & & -468
		\end{array}
		$$	
		If we choose a vector valued modular form of weight $-1 = 1-n/2$ with principal part $e(-\tau/3)\e_{\gamma} + e(-\tau/3)\e_{-\gamma}$ for some $\gamma \in D$ with $Q(\gamma) \equiv -1/3 \mod 1$, then the weight of its Borcherds product will be
		\[
		-\frac{1}{4}(c(\gamma,-1/3)q(\gamma,1/3) + c(-\gamma,-1/3)q(-\gamma,1/3)) = 1
		\]
		which is the singular weight for this lattice. To see that this is essentially the only possibility to obtain singular weight, we will show that 
		\begin{enumerate}
			\item $q(\gamma,m) < -2$ for $\gamma \neq 0$ and $m \geq \frac{2}{3}$ with $m = -Q(\gamma )\mod 1$,
			\item $q(0,m) = 0$ for $m \equiv 2 \mod 3$,
			\item $q(0,m) < -4$ for $m \equiv 0,1 \mod 3$.
		\end{enumerate}
		By (\ref{DivisorSumEstimate}) we can estimate
		\begin{align*}
		q(\gamma,m) \leq - 2(3m)^{2}(2-\zeta(2))
		\end{align*}
		for $\gamma \neq 0$, which is smaller than $-2.8$ for $m \geq \frac{2}{3}$.
		
		For $m \equiv 2 \mod 3$ we have $\chi(m/d) = \chi(m)\chi(d) = -\chi(d)$ and hence $q(0,m) = 0$ in this case. This means the following: Suppose we already have a vector valued modular form $F$ which lifts to an automorphic product $\Psi$ of singular weight. If we alter the coefficient $c(0,m)$ for $m < 0, m \equiv 2 \mod 3$, the weight formula (\ref{weightautomorphicproduct}) shows that the weight of $\Psi$ will not change. Moreover, there are no vectors $x \in L$ with $q(x) = 2\mod 3$, so the zeros of $\Psi$ will also remain the same. Hence $\Psi$ will only vary by a constant factor if we manipulate the coefficients $c(0,m)$ for $m < 0, m\equiv 2 \mod 3$. Therefore we may as well choose $c(0,m) = 0$ for $m < 0, m\equiv 2 \mod 3$.
		
		For $m \equiv 1\mod 3$ we get, using (\ref{DivisorSumEstimate}):
		\begin{align*}
		q(\gamma,m) = -36\sum_{d|m}\chi(m/d)d^{2} \leq -36m^{2}(2-\zeta(2)) < -12.
		\end{align*}
		Now let $m \equiv 0 \mod 3$. For every $d \mid m$ at least one of $\chi(d)$ or $\chi(m/d)$ is zero. Thus the two sums in $q(0,m)$ may be combined in the form
		\begin{align*}
		q(0,m) = -18\sum_{d|m}a_{d}d^{2}
		\end{align*}
		with $a_{d} \in \{-1,0,1\}, a_{m} = 1$. Estimating as above, we get
		\begin{align*}
		q(0,m) \leq -18m^{2}(2-\zeta(2)) < -6.
		\end{align*}
		This shows that there is essentially only one possibility to obtain an automorphic product of singular weight for $L$, at least if we restrict to non-negative principal parts.
		
Now let $L$ be the simple lattice of level $6$, signature $(2,4)$ and discriminant group $D \cong (\mathbb{Z}/2\mathbb{Z})^{4} \times \mathbb{Z}/3\mathbb{Z}$.
We only estimate the formula for $q(0,m)$ since the calculations for $q(\gamma,m)$ with $\gamma \neq 0$ are very similar to the ones above. The formula from \cite{ScheithauerClassification}, Theorem 7.1, for the coefficient $q(0,m)$ of the Eisenstein series simplifies to
		\begin{align*}
	&-\sum_{d|m}\left(36\chi(m/d) - 18\psi(m/d)(-1)^{d}+4\left( \frac{m/d}{2}\right)^{2}\psi(d) -2\psi(d)(-1)^{d}\right)d^{2}.
	\end{align*} 
	where $\psi(d) = \left( \frac{d}{3}\right),  \chi(d) = \left( \frac{d}{12}\right)= \left( \frac{d}{2}\right)^{2}\left( \frac{d}{3}\right)$ are Kronecker symbols. For odd $m$ the formula becomes
	\begin{align*}
	q(0,m) &= -\sum_{d|m}54\psi(m/d)d^{2}-6\sum_{d|m}\psi(d)d^{2}
	\end{align*}
	which can be estimated by
	\begin{align*}
	q(0,m) &\leq - 54(2-\zeta(2))m^{2} + 6\zeta(2)m^{2} < -9m^{2}.
	\end{align*}

	For even $m$ the coefficient $q(0,m)$ is of the form
	\begin{align*}
	q(0,m) = -18\sum_{d|m}a_{d}\psi(m/d)d^{2} - 2\sum_{d|m}a_{d}\psi(d)d^{2}
	\end{align*}
	where $a_{d}$ equals $1$ if $d$ is odd or $\left(\frac{m/d}{2}\right)^{2} = 1$, and $-1$ otherwise. We can estimate
	\begin{align*}
	q(0,m) \leq -18(2-\zeta(2))m^{2} + 2\zeta(2)m^{2} < -3m^{2}
	\end{align*}
	for even $m$. We see that $q(0,m) < -9$ for all $m \in \mathbb{N}$. By (\ref{weightautomorphicproduct}), the weight of an automorphic product coming from a vector valued modular form which has a positive Fourier coefficient $c(0,m) > 0$ for some $m < 0$ must be bigger than $\frac{9}{4}c(0,m) > 2$, i.e. bigger than the singular weight $1$.		
\end{proof}

	We remark that in the prime level case we can also show that allowing negative Fourier coefficients $c(\gamma,m) < 0$ in the principal part of a vector valued modular form $F$ does not lead to new holomorphic automorphic products of singular weight for the simple lattices. The proof is quite tedious as the coefficients $q(\gamma,m)$ of the Eisenstein series need to be estimated much more carefully. Moreover, the estimates for the two simple lattices of level $6$ seem intractable if we allow negative coefficients in the principal part. Hence, for simplicity, we restricted to non-negative principal parts.

%
%
%
%
%

	\section{Explicit construction of automorphic products of singular weight}
	
	The vector valued modular forms from Theorem \ref{thm:ProdukteAufEinfachenGittern} which lead to automorphic products of singular weight are implicitly given by their weights and principal parts. Their existence is guaranteed since the lattices are simple. We want to construct those vector valued modular forms more explicitly as lifts of elliptic modular forms, namely as lifts of eta products for $\Gamma_{1}(N)$.
	
	It follows from the action of $\Gamma_{1}(N)$ in the Weil representation that the component functions $f_{\gamma}$ of a vector valued modular form are weakly holomorphic elliptic modular forms for $\Gamma_{1}(N)$ and character $\chi_{\gamma}\left( M\right) = e(bQ(\gamma))$.
	Here weakly holomorphic means that $f_{\gamma}$ may have poles at the cusps. Conversely, such scalar valued modular forms $f$ of weight $k$ for $\Gamma_{1}(N)$ and character $\chi_{\gamma}$ may be lifted via

			\[
			F_{\Gamma_{1}(N),f,\gamma}(\tau) = \sum_{M \in \Gamma_{1}(N) \setminus \SL_{2}(\mathbb{Z})}f|_{k}M(\tau)\rho_{L}(M^{-1})\e_{\gamma} 
			\]
			to vector valued modular forms for $\rho_{L}$ of weight $k$ (see \cite{ScheithauerModularforms}, Theorem 3.1). To find suitable inputs we consider eta products:
	
	\begin{proposition}\label{EtaProducts}
		Let $N$ be a positive integer. Take integers $r_{\delta}$ for $\delta \mid N$ such that $\frac{N}{24}\sum_{\delta|N} \delta r_{\delta}$ and $\frac{N}{24}\sum_{\delta | N}r_{\delta}/\delta$ are integers and $\sum_{\delta|N}r_{\delta}$ is even. Then the eta product
		$$
		\prod_{\delta | N}\eta(\delta \tau)^{r_{\delta}}
		$$
		is a modular form for $\Gamma_{1}(N)$ of weight $k = \sum_{\delta|N}r_{\delta}/2$ and character
		$$
		\chi\left( \begin{pmatrix} a & b \\ c & d \end{pmatrix}\right) = e\left(b \frac{1}{24} \sum_{\delta|N}\delta r_{\delta} \right).
		$$
	\end{proposition}
	
	\begin{proof}
		The proof is very similar to that of Theorem 6.2 in \cite{BorcherdsReflectionGroups}, so we only give a short sketch: First, one shows (as in \cite{BorcherdsReflectionGroups}, Lemma 5.1) that $\Gamma_{1}(N)$ is generated by the matrices $\left(\begin{smallmatrix} a & b \\ c & d \end{smallmatrix}\right) \in \Gamma_{1}(N)$ with $c > 0, d > 0$ and $d \equiv 1 \mod 4$. Now it suffices to check that the eta product transforms correctly under these generators. Using Rademacher's formula for the transformation behaviour of $\eta$, this is a straightforward calculation analogously to the calculation in \cite{BorcherdsReflectionGroups}, Theorem 6.2.
	\end{proof}	
	
	We remark that Scheithauer also used eta products to construct vector valued modular forms for $\rho_{L}$ of weight $1-b^{-}/2$ leading to automorphic products of singular weight (see \cite{ScheithauerWeil}, for example). Moreover, eta products can be lifted to cusp forms of weight $1+n/2$ for $\rho_{L}^{*}$. In view of (\ref{obstruction}) this can sometimes be useful to show that certain vector valued modular forms with a desired principal part do not exist.
	
	In Theorem \ref{thm:ProdukteAufEinfachenGittern} we found four lattices which allow automorphic products of singular weight. Borcherds' automorphic product of singular weight $12$ for $\II_{2,26}$ is well known. The automorphic product for the lattice of level $2$, signature $(2,10)$ and discriminant group $(\mathbb{Z}/2\mathbb{Z})^{2}$ has already been described in \cite{Borcherds}, example 13.7, and \cite{ScheithauerClassification}, section 8, so we will concentrate on the two new automorphic products. It turns out that the input functions for the singular theta correspondence can be realized as eta products.
	
\begin{theorem}\label{thm:Konstruktion} \quad
\begin{enumerate}
\item
The vector valued modular form corresponding to the case $D=\II_{2,4}(3^{+5})$ can be obtained as $F=F_{f,\Gamma_1(3),\gamma}$  
where 
\[
f(\tau)=\frac{\eta(\tau)}{\eta(3\tau)^3}
\] 
and $\gamma\in D$ is an element of norm $Q(\gamma)\equiv -1/3\mod 1$. 
\item The vector valued modular form corresponding to the case $D=\II_{2,6}(2_{\II}^{-6})$ can be obtained as $F=F_{f,\Gamma_1(2),\gamma}$ where
\[
f(\tau)=\frac{\eta(\tau)^4}{\eta(2\tau)^8}
\] 
and $\gamma\in D$ is an element of norm $Q(\gamma)\equiv 1/2\mod 1$. 
\end{enumerate}
\end{theorem}
\begin{proof}\quad
\begin{enumerate}
\item
We may take 
	$$
	L = A_{2}(-1) \oplus I\!I_{1,1}(3) \oplus I\!I_{1,1}(3).
	$$
	Choose $\gamma \in D = L'/L$ with $Q(\gamma) \equiv -1/3 \mod 1$. By Theorem \ref{EtaProducts} the eta product
	$$
	f(\tau) = \frac{\eta(\tau)}{\eta(3\tau)^{3}} = q^{-1/3} - q^{2/3} - q^{5/3} + 3q^{8/3} - 3q^{11/3} - 2q^{14/3} + 9q^{17/3} + \dots
	$$
	(here $q := e^{2\pi i \tau}$) is a modular form for $\Gamma_{1}(3)$ of weight $-1$ and character $\chi_{\gamma}(M) = e(-b/3)$. Hence we can use it as an input for the $\Gamma_{1}(3)$-lift. The group $\Gamma_{1}(3)$ has two classes of cusps which are represented by $\infty$ and $0$. The expansion of $f$ at the cusp $0$ of $\Gamma_{1}(3)$ is
	\begin{align*}
	f|_{-1}S(\tau) &= 3\sqrt{3}i\frac{\eta(\tau)}{\eta(\tau/3)^{3}} \\
	&= 3\sqrt{3}i(1 + 3q^{1/3} + 9q^{2/3} + 21q + 48q^{4/3} + 99q^{5/3} + 198q^{2} + \dots) \\
	&= g_{0} + g_{1/3} + g_{2/3}.
	\end{align*}
	Here $g_{j/3}$ has only the Fourier coefficients with $q$-exponents equivalent to $j/3 \mod 1$.
	
	 Using the $8$ matrices $\pm I, \pm S, \pm ST, \pm ST^{-1}$ as representatives for $\Gamma_{1}(3)\setminus \SL_{2}(\mathbb{Z})$ or the formula from \cite{ScheithauerModularforms}, Theorem 3.7, we obtain the $\Gamma_{1}(3)$-lift of $f$ on $\e_{\gamma}$ as
	\begin{align*}
	F = F_{\Gamma_{1}(3),f,\gamma} &= f\e_{\gamma} + f\e_{-\gamma} - \frac{3i}{\sqrt{243}}\sum_{\beta \in D}[e((\beta,\gamma)) + e(-(\beta,\gamma))]g_{Q(\beta)}\e_{\beta}.
	\end{align*}
	Thus $F$ is a vector valued modular form for $\rho_{L}$ of weight $-1$ and principal part 
	$$
	e(-\tau/3)\e_{\gamma} + e(-\tau/3)\e_{-\gamma}.
	$$
	Its constant coefficient is $c(0,0) = 2$, hence its Borcherds product $\Psi$ has singular weight $1$. 
\item 
 We may take
	\begin{align*}
	L = D_{4}(-1) \oplus \II_{1,1}(2) \oplus \II_{1,1}(2).
	\end{align*}
	Let $\gamma \in D$ with $Q(\gamma) \equiv 1/2 \mod 1$. The eta product
	\begin{align*}
	f(\tau) = \frac{\eta(\tau)^{4}}{\eta(2\tau)^{8}} = q^{-1/2} -4q^{1/2} + 10q^{3/2} - 24q^{5/2} + 55q^{7/2}-116q^{9/2} + \dots
	\end{align*}
	is a modular form for $\Gamma_{1}(2)$ of weight $-2$ and character $\chi_{\gamma}(M) = e(-b/2)$. Note that $\Gamma_{1}(2)$ has two classes of cusps represented by $\infty$ and $0$. The expansion of $f$ at $0$ is given by
	\begin{align*}
	f|_{-2}S(\tau) &= -16\frac{\eta(\tau)^{4}}{\eta(\tau/2)^{8}} = -16(1 + 8q^{1/2} + 40q + 160q^{3/2} + 552q^{2} + \dots) \\
	&= g_{0} + g_{1/2}.
	\end{align*}
	Using the three matrices $I,S,ST$ as representatives for $\Gamma_{1}(2)\setminus \SL_{2}(\mathbb{Z})$, the $\Gamma_{1}(2)$-lift of $f$ at $\e_{\gamma}$ becomes
	\begin{align*}
	F = F_{\Gamma_{1}(2),f,\gamma} = f\e_{\gamma} -\frac{1}{4}\sum_{\beta \in D}e((\beta,\gamma))g_{Q(\beta)}\e_{\beta}.
	\end{align*}
	Therefore $F$ is a vector valued modular form for $\rho_{L}$ of weight $-2$ with principal part $e(-\tau/2)\e_{\gamma}$. Its constant coefficient is $4$, so the Borcherds product $\Psi$ of $F$ has singular weight $2$.
\end{enumerate}
\end{proof}
	
 We compute Fourier and product expansions of the two automorphic products at different cusps. For the general product expansion of an automorphic product and the notation in this context the reader should consult \cite{BruinierBorcherdsProducts}.

	For the first of the new automorphic products (corresponding to item 1 in the previous theorem) let $L=A_{2}(-1) \oplus I\!I_{1,1}(3) \oplus I\!I_{1,1}(3)$. Let $z \in L$ be a cusp, i.e. a primitive isotropic vector in $L$. Then $z$ has level $3$ (i.e. $\langle z,L\rangle = 3\mathbb{Z}$), since otherwise $L$ would split a hyperbolic plane $\II_{1,1}$ which is not possible. We compute the product expansion of $\Psi_{z}$ (cf. \cite{BruinierBorcherdsProducts}, Theorem 3.22): Choose another primitive isotropic vector $\zeta \in L$ with $\langle z,\zeta \rangle = 3$ such that $L$ decomposes as an orthogonal direct sum $L = K \oplus \II_{1,1}(3)$ where the hyperbolic plane $\II_{1,1}(3)$ is generated by $z$ and $\zeta$. Then $K = A_{2}(-1)\oplus \II_{1,1}(3)$ is a lattice of signature $(1,3)$.  Let $z' = \zeta/3 \in L'$. Since $\zeta$ is orthogonal to $K$ the projection $\zeta_{K}$ is $0$ and $p(\lambda) = \lambda_{K}$ for $\lambda \in L_{0}'$. For $\lambda \in K'$ a system of representatives of the $\delta \in L_{0}'/L$ with $p(\delta) = \lambda + K$ is therefore given by $\lambda, \lambda + z/3, \lambda + 2z/3 \in L_{0}'$. We obain the product expansion
		\begin{align*}
		\Psi_{z}(Z) = e\big((\rho,Z) \big) \prod_{\substack{\lambda \in K' \\ (\lambda,W) > 0}}&(1-e((\lambda,Z)))^{c(\lambda+L,q(\lambda))} \\
		\times &(1-e(1/3)e((\lambda,Z)))^{c(\lambda + z/3+L,q(\lambda))} \\
		\times &(1-e(2/3)e((\lambda,Z)))^{c(\lambda + 2z/3+L,q(\lambda))},
		\end{align*}
		where $W$ is a Weyl chamber of the positive cone $C$ of $K \otimes \mathbb{R}$ (which is one of the connected components of the set of positive norm vectors in $K \otimes \R$), $\rho$ is the corresponding Weyl vector and $c(\delta,n)$ are the coefficients of the vector valued modular form $F$.
		
		Now we determine Fourier expansions of $\Psi_{z}$ at different cusps. We fix $z$ and consider different choices of $\gamma$.
		
		First, let $\gamma = z/3 - \zeta/3 +L$. Then $Q(\gamma) \equiv - 1/3 \mod 1$ and $(\gamma,z) \equiv 1 \mod 3$. Moreover, we have $(\gamma,\lambda) \equiv 0 \mod 1$ for all $\lambda \in K'$. Let $F$ be the vector valued modular form for $\rho_{L}$ of weight $-1$ and principal part $e(-\tau/3)\e_{\gamma}$ $+ e(-\tau/3)\e_{-\gamma}$ as constructed above and let $\Psi$ be the corresponding automorphic product. We obtain the following Fourier expansion of $\Psi_{z}$:
		
		\begin{theorem}\label{FourierEntwicklungNot0}
		Let $z, \zeta \in L$ be primitive isotropic vectors with $\langle z,\zeta \rangle = 3$ such that $L = K \oplus \II_{1,1}(3)$ as an orthogonal direct sum where the hyperbolic plane is generated by $z$ and $\zeta$, and $K = A_{2}(-1) \oplus \II_{1,1}(3)$. Choose $\gamma = z/3 -\zeta/3 + L \in L'/L$. Define
		\begin{align*}
		f_{0}(\tau) &= \frac{1}{3\sqrt{3}i}f|_{-1}S(\tau) = \frac{\eta(\tau)}{\eta(\tau/3)^{3}} \\ 
		&= 1 + 3q^{1/3} + 9q^{2/3} + 21q + 48q^{4/3} + 99q^{5/3}+198q^{2} + \dots
		\end{align*}
		Let $C$ be the positive cone in $K \otimes \mathbb{R}$ and $K'^{+} = (K' \cap \overline{C}) \setminus \{0\}$. Then 
		\begin{align*}
		\Psi_{z}(Z) &= \prod_{\lambda \in K'^{+}}\frac{\big(1-e((\lambda,Z))\big)^{3[f_{0}](q(\lambda))}}{\big(1-e(3(\lambda,Z))\big)^{[f_{0}](q(\lambda))}} = 1 + \sum_{\lambda \in K'^{+}}c(\lambda)e((\lambda,Z)),
		\end{align*}
		where $c(\lambda)$ is $0$ unless $\lambda = n\mu$ for some primitive isotropic $\mu \in K'^{+}$, in which case $c(\lambda)$ equals the coefficient at $q^{n}$ in
		$$
		\frac{\eta(\tau)^{3}}{\eta(3\tau)} = 1 -3q + 6q^{3} - 3q^{4} - 6q^{7} +6q^{9} + 6q^{12} +  \dots
		$$
	\end{theorem}
	
	We remark that this equation was already found by Scheithauer in \cite{ScheithauerTwisting}, Proposition 6.10, as a twisted denominator identity of the fake monster superalgebra.

	\begin{proof}
		Since $\gamma \notin L_{0}'$ the Weyl vector $\rho$ is $0$ and the only Weyl chamber is the positive cone $C$. The set of elements $\lambda \in K'$ with $(\lambda,C) > 0$ equals $K'^{+}$. For $\lambda \in K'^{+}$ we have
		\begin{align*}
		c(\lambda,q(\lambda)) &= \big(e((\lambda,\gamma)) + e(-(\lambda,\gamma))\big)[f_{0}](q(\lambda)) = 2[f_{0}](q(\lambda))
		\end{align*}
		since $(\lambda,\gamma) \equiv 0 \mod 1$. Analogously we obtain
		\begin{align*}
		c(\lambda+z/3,q(\lambda)) &= c(\lambda+2z/3,q(\lambda)) =  - [f_{0}](q(\lambda)).
		\end{align*}
		Thus the product expansion of $\Psi_{z}$ is given by
		\begin{align*}
		\Psi_{z}(Z) &= \prod_{\lambda \in K'^{+}}\frac{(1-e((\lambda,Z)))^{2[f_{0}](q(\lambda))}}{\big((1-e(1/3)e((\lambda,Z)))(1-e(2/3)e((\lambda,Z)))\big)^{[f_{0}](q(\lambda))}} \\
		&= \prod_{\lambda \in K'^{+}}\frac{\big(1-e((\lambda,Z))\big)^{3[f_{0}](q(\lambda))}}{\big(1-e(3(\lambda,Z))\big)^{[f_{0}](q(\lambda))}}.
		\end{align*}
		
		Now we compute the Fourier expansion of $\Psi_{z}$. Since $\Psi_{z}$ has singular weight, it has a Fourier expansion of the form
		$$
		\Psi_{z}(Z) = 1 + \sum_{\lambda \in K'^{+}}c(\lambda)e((\lambda,Z))
		$$
		where $c(\lambda) \neq 0$ implies $q(\lambda) = 0$. This allows us to determine the coefficients $c(\lambda)$ in a combinatorial way: Let $\lambda \in K'^{+}$ with $q(\lambda) = 0$ and suppose $\lambda = \sum \lambda_{i}$ with $\lambda_{i} \in K'^{+}$, i.e. the $\lambda_{i}$ contribute to $c(\lambda)$ in the product expansion of $\Psi_{z}$. Then $0 = \langle \lambda, \lambda \rangle = \sum\langle \lambda,\lambda_{i}\rangle$ and $\langle \lambda,\lambda_{i} \rangle \geq 0$ imply $\langle \lambda,\lambda_{i} \rangle = 0$. This means that $\lambda$ and the $\lambda_{i}$ are positive integral multiples of the same primitive norm 0 vector $\mu \in K'^{+}$. Therefore contributions to $c(\lambda)$ in the product expansion come from
		\begin{align*}
		\prod_{m > 0}\frac{\big(1-e((m\mu,Z))\big)^{3[f_{0}](q(m\mu))}}{\big(1-e(3(m\mu,Z))\big)^{[f_{0}](q(m\mu))}} =
		\prod_{m > 0}\frac{\big(1-e((m\mu,Z))\big)^{3}}{1-e(3(m\mu,Z))} = \frac{\eta((\mu,Z))^{3}}{\eta(3(\mu,Z))}.
		\end{align*}
		This shows that for $\lambda = n\mu$ the value $c(\lambda)$ is given by the coefficient at $q^{n}$ of $\eta(\tau)^{3}/\eta(3\tau)$.
	\end{proof}

	We keep the vectors $z,\zeta \in L$ and the decomposition $L = K\oplus \II_{1,1}(3)$ as above, where the hyperbolic plane is generated by $z,\zeta$. Now take two primitive isotropic vectors $x,\xi \in K$ with $\langle x,\xi \rangle = 3$ which span the hyperbolic plane in $K = A_{2}(-1) \oplus \II_{1,1}(3)$ and set $\gamma = x/3 - \xi/3 + L$. Then we have $Q(\gamma) \equiv -1/3 \mod 1$ and $(z,\gamma) \equiv 0 \mod 3$. We write $G$ for the subgroup of $O(K)$ which is generated by the reflections
	$$
	\sigma_{\alpha}(x) = x- \frac{(\alpha,x)}{q(\alpha)}\alpha
	$$
	along vectors $\alpha \in K'$ with $q(\alpha) = -1/3$ and $\alpha + K  = p(\gamma) = x/3-\xi/3 +K$ and call $G$ the Weyl group of $F$. The Weyl chambers of $F$ are the connected components of the positive cone $C$ after removing the reflection hyperplanes $\alpha^{\perp}$ (see \cite{BruinierBorcherdsProducts}, definitions before Theorem 3.22). 
	
	\begin{theorem}
		Let $x, \xi \in K = A_{2}(-1) \oplus \II_{1,1}(3)$ be primitive isotropic vectors with $\langle x,\xi \rangle = 3$ and $\gamma = x/3 - \xi/3 + K$. Choose some Weyl chamber $W$ whose closure contains $x$. Then the Weyl vector $\rho = \frac{1}{3}x$ is a primitive isotropic vector of $K'$ and $\Psi_{z}$ has the Fourier expansion
		\begin{align*}
		\Psi_{z}(Z)
		& = \sum_{w \in G}\det(w)\frac{\eta(9(w(\rho),Z))^{3}}{\eta(3(w(\rho),Z))},
		\end{align*}
		where $G$ is the Weyl group of $F$, i.e. the subgroup of $O(K)$ generated by the reflections along the vectors
		$$
		\alpha \in K' \text{ with } q(\alpha) = -1/3 \text{ and } \alpha + K = p(\gamma).
		$$
	\end{theorem}
	
	\begin{proof}
		The Weyl vector $\rho = \frac{1}{3}x$ may be computed with the formula from \cite{Borcherds}, Theorem 10.4 (note that there is a mistake in the formula for $\rho_{z}$ which Borcherds corrects in the introduction of \cite{BorcherdsReflectionGroups}: the condition $(\lambda,W) > 0$ has to be dropped and the factor $\frac{1}{2}$ in front of the sum has to be replaced by $\frac{1}{4}$).
		
	The product $\Psi_{z}$ is antisymmetric under the Weyl group $G$, i.e. $\Psi_{z}(w(Z)) = \det(w)\Psi_{z}(Z)$ for all $w \in G$. Since $G$ acts simply transitively on the Weyl chambers of the positive cone of $K \otimes\mathbb{R}$, we find that $\Psi_{z}$ has a Fourier expansion of the form
	$$
	\Psi_{z}(Z) = \sum_{w\in G}\det(w)\sum_{\lambda + \rho \in \overline{W}}c(\lambda)e((w(\lambda+\rho),Z))
	$$
	where $c(\lambda) \neq 0$ is only possible if $q(\lambda + \rho) = 0$ because $\Psi_{z}$ has singular weight. Let $\lambda \in K'$ with $(\lambda,W) > 0$ and $\lambda + \rho \in \overline{W}$. We want to determine $c(\lambda)$. We have
	$$
	\langle \lambda,\rho \rangle = q(\lambda + \rho) - q(\lambda) - q(\rho) = -\langle \lambda,\lambda \rangle/2
	$$
	and, since $\langle \lambda,\lambda + \rho \rangle \geq 0$ ,
	$$
	\langle \lambda,\rho \rangle = \langle \lambda, \lambda + \rho \rangle - \langle \lambda,\lambda \rangle \geq - \langle \lambda,\lambda \rangle.
	$$
	It follows $\langle\lambda,\rho\rangle \leq 0$. On the other hand, we have $\langle \lambda, \rho \rangle \geq 0$ because $\rho \in \overline{W}$ and $(\lambda,W) > 0$. Together we have $\langle\lambda,\rho \rangle = 0, q(\lambda) = 0$ and $\langle \lambda + \rho,\lambda \rangle = 0$. Since both $\rho$ and $\lambda + \rho$ lie in the closure of the positive cone, the condition $\langle \lambda,\lambda+\rho \rangle =0$ means that both are positive integral multiples of the same primitive norm 0 vector in $K'$, hence multiples of $\rho$. Let $\lambda + \rho = n\rho$. 
	
	Suppose $n\rho = \sum \lambda_{i}$ with $\lambda_{i} \in K'$ and $(\lambda_{i},W) > 0$, i.e. the coefficients $\lambda_{i}$ contribute to $c(\lambda)$ in the product expansion of $\Psi_{z}$. Then $0 = \langle n\rho,\rho\rangle = \sum\langle \lambda_{i},\rho \rangle$ and $\langle \lambda_{i},\rho \rangle \geq 0$ show that $\langle \lambda_{i},\rho \rangle = 0$. 
	Assume $q(\lambda_{i}) < 0$. From $c(\lambda_{i}+L,q(\lambda)) \neq 0$ it follows that $\lambda_{i} + L = \pm\gamma + L$ and $\langle \lambda_{i},\rho\rangle \equiv \langle \pm \gamma,\rho \rangle \equiv \pm \frac{1}{3} \mod 1$, a contradiction. Hence $q(\lambda_{i}) \geq 0$ and the $\lambda_{i}$ are positive integral multiples of $\rho$. Therefore contributions in the product expansion of $\Psi_{z}$ to $c(\lambda)$ come from
	\begin{align*}
	e((\rho,Z))\prod_{m > 0}\big(1-e((m\rho,Z)) \big)^{c(m\rho,0)}&\big(1-e(1/3)e((m\rho,Z))\big)^{c(m\rho+z/3,0)} \\
	& \times\big(1-e(2/3)e((m\rho,Z))\big)^{c(m\rho+2z/3,0)}.
	\end{align*}
	We have
	$$
	c(m\rho,0) = c(m\rho+z/3,0) = c(m\rho + 2z/3,0) = e(m(\rho,\gamma)) + e(-m(\rho,\gamma)) 
	$$
	and $(\rho,\gamma) = -1/3 \mod 1$. Hence the exponents are $2$ if $m \equiv 0 \mod 3$ and $-1$ otherwise. Thus the product becomes
	\begin{align*}
	e((\rho,Z))&\prod_{\substack{m > 0 \\ m\not\equiv 0 \modulo 3}}(1-e(3(m\rho,Z)))^{-1} \prod_{\substack{m > 0 \\ m \equiv 0 \modulo 3}} (1-e(3(m\rho,Z)))^{2} \\
	&= e((\rho,Z))\prod_{m > 0}\frac{(1-e(9(m\rho,Z)))^{3}}{1-e(3(m\rho,Z))} = \frac{\eta(9(\rho,Z))^{3}}{\eta(3(\rho,Z))}.
	\end{align*}
	This shows 
	$$
	\sum_{\lambda +\rho \in \overline{W}}c(\lambda)e((\lambda+\rho,Z)) = \frac{\eta(9(\rho,Z))^{3}}{\eta(3(\rho,Z))}. 
	$$
	We obtain the desired Fourier expansion of $\Psi_{z}$.
	\end{proof}
	
	Expansions of the automorphic product of singular weight for the lattice $L=D_{4}(-1) \oplus \II_{1,1}(2) \oplus \II_{1,1}(2)$ can be computed in a similar way. The proofs are almost the same, so we give only the results:
	
	The product expansion of $\Psi_{z}$ at a cusp $z \in L$ is given by
	\begin{align*}
	\Psi_{z}(Z) = e((\rho,Z))\prod_{\substack{\lambda \in K' \\ (\lambda,W) > 0}}(1-e((\rho,Z)))^{c(\lambda,q(\lambda))}(1+e((\rho,Z)))^{c(\lambda + z/2,q(\lambda))},
	\end{align*}
	where $K$ comes from the decomposition $L = K \oplus \II_{1,1}(2)$ in which $\II_{1,1}(2)$ is generated by $z$ and a primitive isotropic $\zeta \in L$ with $\langle z,\zeta \rangle = 2$.
	
	We fix the cusp $z$ and consider different choices for $\gamma \in D$ with $Q(\gamma) \equiv -1/2 \mod 1$. First, let $\gamma = z/2 - \zeta/2 + L$. Then it holds $(\lambda,\gamma) \equiv 0 \mod 1$ for all $\lambda \in K'$ and $\gamma \notin L_{0}'$. Therefore the Weyl vector $\rho$ is $0$ and the only Weyl chamber is the positive cone $C$. We obtain:

		\begin{theorem}
		Let $z, \zeta \in L$ be primitive isotropic vectors with $\langle z,\zeta \rangle = 2$ such that $L = K \oplus \II_{1,1}(2)$ as an orthogonal direct sum where the hyperbolic plane is generated by $z$ and $\zeta$, and $K = D_{4}(-1) \oplus \II_{1,1}(2)$. Choose $\gamma = z/2 -\zeta/2 + L \in L'/L$. Define
		$$
		f_{0}(\tau) = -\frac{1}{16}f|_{-1}S(\tau) = \frac{\eta(\tau)^{4}}{\eta(\tau/2)^{8}} = 1 + 8q^{1/2} + 40q + 160q^{3/2} + 552q^{2} + \dots
		$$
		Let $C$ be the positive cone in $K \otimes \mathbb{R}$ and $K'^{+} = (K' \cap \overline{C}) \setminus \{0\}$. Then 
		\begin{align*}
		\Psi_{z}(Z) &= \prod_{\lambda \in K'^{+}}\frac{\big(1-e((\lambda,Z))\big)^{4[f_{0}](q(\lambda))}}{\big(1+e((\lambda,Z))\big)^{4[f_{0}](q(\lambda))}} = 1 + \sum_{\lambda \in K'^{+}}c(\lambda)e((\lambda,Z)),
		\end{align*}
		where $c(\lambda)$ is $0$ unless $\lambda = n\mu$ for some primitive isotropic $\mu \in K'^{+}$, in which case $c(\lambda)$ equals the coefficient at $q^{n}$ in
		$$
		\frac{\eta(\tau)^{8}}{\eta(2\tau)^{4}} = 1 -8q + 24q^{2} - 32q^{3} + 24q^{4} -48q^{5} + 96q^{6} +  \dots
		$$
	\end{theorem}
	
	This equation is the denominator identity of a generalized Kac-Moody algebra and was also found by Scheithauer in \cite{ScheithauerWeil}, Theorem 8.1, as the level $4$ expansion of an automorphic product of singular weight for the lattice $D_{4}(-1) \oplus \II_{1,1} \oplus \II_{1,1}(4)$. The corresponding vector valued modular form can be constructed as the lift of the eta product $2\eta(2\tau)^{4}/\eta(\tau)^{8}$.
	
	We keep the cusp $z$ and the corresponding decomposition $L = K \oplus \II_{1,1}(2)$ fixed and choose primitive isotropic vectors $x,\xi \in K$ with $\langle x,\xi \rangle = 2$. Further, we set $\gamma = x/2-\xi/2 + L$. Then $Q(\gamma) \equiv -1/2 \mod 1$ and $\gamma \in L_{0}'$. The Weyl chambers are the connected components of the positive cone after removing the hyperplanes $\alpha^{\perp}$ for $\alpha \in K'$ with $q(\alpha) = -1/2$ and $\alpha + K = p(\gamma) = x/2 - \xi/2 +K$. 
	
	\begin{theorem}
		Let $x, \xi \in K = D_{4}(-1) \oplus \II_{1,1}(2)$ be primitive isotropic vectors with $\langle x,\xi \rangle = 2$ and $\gamma = x/2 - \xi/2 + K$. Choose some Weyl chamber $W$ whose closure contains $x$. Then the Weyl vector $\rho = \frac{1}{2}x$ is a primitive isotropic vector of $K'$ and $\Psi_{z}$ has the Fourier expansion
		\begin{align*}
		\Psi_{z}(Z)
		& = \sum_{w \in G}\det(w)\frac{\eta(4(w(\rho),Z))^{8}}{\eta(2(w(\rho),Z))^{4}},
		\end{align*}
		where $G$ is the Weyl group of $F$, i.e. the subgroup of $O(K)$ generated by the reflections along the vectors
		$$
		\alpha \in K' \text{ with } q(\alpha) = -1/2 \text{ and } \alpha + K = p(\gamma).
		$$
	\end{theorem}

	\bibliography{references}{}
	\bibliographystyle{plain}

\end{document}